\documentclass[12pt]{article}

\usepackage{amssymb,latexsym,amsmath} % Standard packages
\usepackage{graphicx}

\newtheorem{theorem}{Theorem}
\newenvironment{proof}{\noindent{\bf Proof:}}{$\hfill \Box$ \vspace{10pt}}

\begin{document}
\title{On the Differential Equation $\frac{d}{dt}\left(\frac{\cos x}{1-\dot{x}}\right)\,=\,-\sin x$ }
\author{ P.S.Datti and B.R.Nagaraj\\
 T.I.F.R. Centre for Applicable Mathematics\\ Bangalore, India\\
\texttt{Email: psd@math.tifrbng.res.in and brn@math.tifrbng.res.in}}
%\date{\today}
\maketitle
\begin{abstract}
The autonomous differential equation in the title is derived in S.Srinivasan~\cite{ss05} (equation (E) in ~\cite{ss05}) in the context of certain discrete sums from the number theoretic considerations. These discrete sums are then estimated in terms of an integral involving the solutions of this differential equation; no analysis is done on this integral in ~\cite{ss05}. Our main objective is to consider this equation from the dynamical systems view point and describe all the orbits in the phase plane. The equation is singular in the sense that the coefficient of $\ddot{x}$ vanishes or becomes unbounded at a few points. This singularity puts a restriction on the initial data. The solutions that we obtain in explicit form are, however, smooth and satisfy the equation pointwise or in the limiting sense. The explicit form of solutions may also be used to analyse the integral in ~\cite{ss05}. The equation possesses an infinite number of equilibrium or critical points and there are periodic orbits surrounding each of them in a part of the phase plane. In the complement of this part of the phase plane, there are non-periodic orbits, any two of which have an infinite number of common points. This, however, is in great contrast with the fact that any two (distinct) orbits of a regular autonomous system are disjoint, thus bringing out the singular nature of the equation.\\

\end{abstract}
\section{Introduction and Statement of the Main Result} The differential equation

\begin{equation} \label{1}
\frac{d}{dt} \left(\frac{\cos x}{1-\dot{x}}\right)\,=\,-\sin x 
\end{equation}
 
\noindent where $\cdot$ denotes the differentiation w.r.t. $t$ and $x$ is the unknown (real) function, was introduced by S.Srinivasan~\cite{ss05} in a number-theoretic context involving certain discrete sums.(He uses the notation $s$ instead of $x$). Our main aim is to study this diffrential equation from the dynamical systems point of view and describe all its orbits. It turns out that this equation has close resemblance with the unforced,undamped pendulum equation $\ddot{x}\,+\,k \sin x\,=\,0\,(k>0)$ (see, e.g., Simmons and Krantz~\cite{sk08}.) Unlike the pendulum equation, the equation (\ref{1}) exhibits solutions in explicit form for all possible initial data. Specifically, we show that the orbits of this equation are periodic iff $\dot{x} (0)\,<\,1/2.$ The apparent singular nature of the equation in the form of the coefficient of the highest order derivative term, namely, $\ddot{x}$, is resolved in an interesting way and we do obtain smooth solutions in explicit form.\\

\noindent Rewrite the equation (\ref{1}) as

\begin{equation} \label{2}
 \frac{\cos x}{(1\,-\, \dot{x})^2 } \ddot{x}\,-\,\frac{\sin x}{1\,-\, \dot{x}} \dot{x}\,=\, -\sin x. 
\end{equation}
We consider the initial value problem by imposing the appropriate initial conditions
\[ x(0)\,=\,a,\,\,\dot{x} (0)\,=\,b ~~\mbox{with}~~ a,\,b\, \in \mathbb{R}. \]
%This equation was introduced by S.Srinivasan[] in the context of a number theoretic problem.
The equation (\ref{2}) is an autonomous equation, but not in the conservative form. %it has many similarities with the undamped,unforced pendulum equation  $\ddot{x}\,+\,k \sin x\,=\,0\,(k>0)$, (see,e.g., Simmons and Krantz \cite{sk08}) including the existence of many periodic orbits. Unlike the pendulum equation, (\ref{2}) is singular and exhibits solutions in explicit form for all possible initial data. 
Our main aim in this paper is to describe the orbits of solutions of (\ref{2}) in the $(x,\dot{x})$ plane, {\it the phase plane}, for all the possible initial data.\\

Looking at the coefficient of $\ddot{x}$ in (\ref{2}), we observe that it is essential to take $a \neq (2n+1)\pi /2, n \in \mathbb{Z}$ and $b \neq 1.$ We will now make several observations regarding the solutions of (\ref{2}) and make the necessary conditions on the initial data $a,b.$
 
\begin{enumerate}
\item The equlibrium points of the equation are $(n \pi\,,\,0)$ with $n$ an integer.
\item If $x$ is a solution, so is $x\,+\,n \pi.$
\item The line  $\dot{x} =1/2,$ in the phase plane is an orbit  of the equation, i.e., $x(t)\,=\,a\,+\,t/2$ is a solution of the equation for any $a \in \mathbb{R}.$
\end{enumerate}

Therefore, the orbits with $b\,<\, 1/2(~\mbox{respectively}~b\,>\, 1/2)$ remain in the region $\dot{x}\,<\, 1/2 (~\mbox{respectively}~\dot{x}\,>\, 1/2)$ in the $(x,\dot{x})$ phase palne. See, also, (\ref{E}) below.\\

With these observations, we assume the following conditions on the initial data:

\begin{eqnarray} \label{I}
|a|\,<\,\pi /2 ~\mbox{and}~ b \neq   \frac{1}{2},1 ~~\mbox{with}~~ |a|\,+\,|b|\,>\,0.
\end{eqnarray}

We now state the main result.\\

\begin{theorem} The following hold:

\begin{enumerate}
\item The equation (\ref{1}) has periodic solutions only when $b\,<\,\frac{1}{2}.$ The orbits in this case are periodic orbits surrounding the origin in the phase plane and are symmetric w.r.t. $\dot{x}$ axis.

\item If $b\,>\,\frac{1}{2},$ then there is an increasing real sequence $\{t_j\},j \in \mathbb{Z},$ whose elements are the roots of a trigonometric equation, with $t_0$ the smallest positive root, such that the solution $x$ satisfies $x(t_j) =$  an odd integer multiple of $\pi /2$ and $\dot{x}(t_j)=1$ for all $j \in \mathbb{Z}.$

\noindent The equation and the energy equation((\ref{E}), derived below) are satisfied at these points in the sense of limit $t \to t_j.$\\

\noindent The orbits in this case oscillate in the phase plane around the 'interface' line $\dot{x}=1$ line, in the strip between the lines $\dot{x}\,=\,\frac{\sqrt{1+c}}{\sqrt{1+c} + 1}$ and $\dot{x}\,=\,\frac{\sqrt{1+c}}{\sqrt{1+c} - 1}$.(Incidentally, $1$ is the harmonic mean of $\frac{\sqrt{1+c}}{\sqrt{1+c} + 1}$ and $\frac{\sqrt{1+c}}{\sqrt{1+c} - 1}.)$ The crossing of the interface line occurs exactly at points corresponding to $t\,=\,t_j.$  
\end{enumerate}

In either case, the solution has an explicit form.
\end{theorem}

\begin{proof}
The proof is a consequence of various relations derived in what follows and many computations done below.
\end{proof}\\

\noindent {\bf Remark} The case $b<1/2$ is similar to the case of pendulum equation for small total energy and that of $b>1/2,\,b \neq 1$ for large total energy. See, Simmons and Krantz \cite{sk08}.\\

Linearizing around an equilibrium point $(n \pi\,,\,0), n \in \mathbb{Z,}$ by writing $x\,=\,y\,+\,n \pi$ with $y$ small, the linearized equation is
\[ \ddot{y} \,=\,-y \]
and hence the equlibrium point $n \pi$ is linearly stable but not asymptotically stable.

\section{First integration. The "energy" equation} Assuming $\cos x \neq 0,$ we may rewrite the equation (\ref{2}) as
\begin{equation}
\ddot{x}\,=\,(\tan x)(1\,-\,\dot{x})(2 \dot{x}\,-\,1). 
\end{equation}
Multiplying this by $\dot{x}$ and integrating once yields 

\begin{equation} \label{E}
(\cos^{2} x) \frac{2 \dot{x}\,-1}{(1\,- \dot{x})^2}\,=\,c,
\end{equation}

\noindent where $c$ is a constant. In analogy with a conservative equation, we call this the "energy"  equation. This energy equation will be used to describe the orbits of the equation and also find the solutions in explicit form. In the process, we also learn how to handle the apparent singularity arising in the equation in the coefficient of $\ddot{x}.$\\

First note that from the initial data, we have $c\,=\,\frac{2b\,-\,1}{(1\,-\,b)^2}\cos^{2} a.$
It immediately follows that, using (\ref{I}), $c \in (-1,0)$ if $b\,<\,1/2$ and $c>0$ for $b\,>\,1/2$ and the sign of $2\dot{x}\,-\,1$ is the same as the sign of $2b\,-\,1$ thus proving the remark made above regarding the orbits of solutions.\\

Taking the $\cos^{2} x$ term the other side and adding and substracting $\dot{x} ^2$ to the numerator term on the left side of (\ref{E}), we obtain
\begin{equation} \label{e}
\frac{\dot{x}^2}{(1\,-\,\dot{x})^2}\,=\,1\,+\,c \sec^{2} x. 
\end{equation}
Using (\ref{e}), we now describe the phase portrait of orbits in the phase plane.\\

Suppose $b\,<\,1/2.$ Then $c\, \in \, (-1,0)$ and therefore the $x-$ values are restricted to a symmetrtic interval $[-\alpha,\,\alpha]$ around $0$ with $\alpha\,<\,\pi /2$ satisfying $\cos (\alpha)\,=\,\sqrt{-c}.$ Consequently, the values of $\dot{x}$ are also restricted to a finite interval $[-\frac{\sqrt{1+c}}{1-\sqrt{1+c}}, \frac{\sqrt{1+c}}{1+\sqrt{1+c}}],$ determined by the energy equation. Using (\ref{e}) we obtain that
\begin{equation} \label{e1}
\dot{x}\,=\,\frac{\pm \sqrt{1+ c \sec^2 x}}{1 \pm \sqrt{1+ c \sec^2 x}},
\end{equation}
with $'+'$ sign if $\dot{x}>0$ and $'-'$ sign if $\dot{x}<0.$\\
This enables us to draw a phase portrait in this case. Thus, we obtain a periodic orbit around the equilibrium point $(0,0),$ traversing clockwise indicating the direction of increasing $t.$ 

Next, suppose $b>1/2$ and $b \neq 1.$ In this case we have $c>0$ and therefore, there is no restriction on $x$ values and $x$ is an increasing function of $t$. Hence the solutions are unbounded in this case. In this case we have,
\begin{equation} \label{e2}
\dot{x}\,=\,\frac{\sqrt{1+ c \sec^2 x}}{\sqrt{1+ c \sec^2 x} \pm 1},
\end{equation}
with $'+'$ sign if $\dot{x} < 1$ and $'-'$ sign if $\dot{x} > 1.$\\

Using (\ref{e2}), we can now complete the phase portarit in this case and we find that the orbit oscillates between the lines $\dot{x}\,=\,\frac{\sqrt{1+c}}{1+\sqrt{1+c}}$ and $\dot{x}\,=\,\frac{\sqrt{1+c}}{\sqrt{1+c}-1}$ in the phase plane, crossing the 'interface' line $\dot{x}=1$ at the points where $x$ is an odd integer multiple of $\pi /2.$\\

Because of the apparent singularity in the equation at those $t$ where $x(t)$ is an odd integer multiple of $\pi /2$ and $\dot{x}(t)=1,$ we will explain the sense in which the equation is satisfied at such $t$ once we obtain the solution in explicit form and see that this singularity does not cause any difficulty.

\section{Final integration. Explicit formula for the solution}. We now perform an integration of the enrgy relation to obtain the solution in explicit form. Looking at the energy relation, it is convenient to introduce the variable $X\,=\, \tan x.$ We then obtain the following equation  taking the positive sign in (\ref{e1}) or (\ref{e2}) (the other cases being similar),

\begin{equation} \label{3}
 \dot{X}\,=\,(1\,+\,X^2) \frac{\sqrt{1+c+c X^2}}{1+ \sqrt{1+c+c X^2}} 
\end{equation}

\noindent Integrating this equation and making use of addition formula for $\arctan$ function we obtain, after some algebra,

\begin{equation} \label{X}
X\,=\, \frac{A \cos t\,+\,B \sin t}{2(1-b)\,+\,B \cos t\,-\, A \sin t}. 
\end{equation} 

\noindent Here $A\,=\,\sin 2a$ and $B\,=\,2b\,-\,1\,+\,\cos 2a.$ We note the following relation connecting $A,B,a,b$ and $c$, which is used in computations that follow. We have,

\begin{equation} \label{D}
A^2\,+\,B^2\,-\,4(1-b)^2\,=\,4(2b-1) \cos^2 a.
\end{equation}

We may rewrite (\ref{D}) as

\begin{equation} \label{D1}
A^2 + B^2\,=\,4(1-b)^2 (1+c).
\end{equation}

Though (\ref{X}) gives the solution in the explicit form, it will be difficult to see whether the solution is periodic or not because of the presence of $\arctan .$ We now obtain the following explicit formula for $\dot{x}$, using (\ref{X}), which will be used to decide the periodicity of the solution easily and also in further analysis. We have:

\begin{eqnarray} \label{xd}
\dot{x} (t) &=& \frac{A^2\,+\,B^2\,+\,2(1-b)(B \cos t\,-\, A \sin t)}{A^2\,+\,B^2\,+\,4(1-b)(B \cos t\,-\, A \sin t)\,+\,4(1-b)^2} \notag \\
&=& \frac{1}{2}\,+\,\frac{2(2b-1) \cos^2 a}{A^2\,+\,B^2\,+\,4(1-b)(B \cos t\,-\, A \sin t)\,+\,4(1-b)^2}
\end{eqnarray}

\noindent Note that the denominator in (\ref{xd}) is positive being the sum of two square terms $(A \cos t + B \sin t)^2 + (2(1-b) + (B \cos t - A \sin t))^2.$ These two terms cannot simultaneouly vanish because of our assumption (\ref{I}). This also implies that the terms in the numerator and denominator of the expression for $X$ in (\ref{X}) do not vanish simultaneously. From (\ref{xd}), we again observe that the sign of $\dot{x} - 1/2$ is solely determined by the sign of $b - 1/2.$ \\

It is not hard to verify the energy relation using the above expressions for $X$ and $\dot{x}.$\\

We now show how to use the above expression for $\dot{x}$ to decide when the solution $x$ will be periodic, with the aid of the following simple lemma, which is straight forward to prove.\\

{\bf Lemma:} Suppose $f: \mathbb{R} \to \mathbb{R}$ is a continuous periodic function of period $T$ and define $g: \mathbb{R} \to \mathbb{R}$ by $g(t)\,=\,\int_{0}^{t}\,f(s)\,ds\,-\,<f> t,$ where $<f>\,=\,\frac{1}{T} \int_{0}^{T}\,f(t)\,dt$ is the average of $f$ over a period interval. Then $g$ is also a periodic function, of period $T.$\\

Using (\ref{xd}) we see that $\dot{x}$ is a periodic function of period $2\pi$ and $<\dot{x}>$ is $0$ or $1$ according as $b\,<\,1/2$ or $b\,>\,1/2$ respectively. Hence by the Lemma, it follows that $x$ is periodic iff $b\,<\,1/2$ and the function $x(t)\,-\,t$ is periodic when $b\,>\,1/2.$ \\

\section{Analysis of (\ref{X}) and (\ref{D})}
We now analyse whether a solution $x$ takes the value $\pi /2$ or an odd integer multiple of it, in finite time and how to handle that situation.\\

Referring to (\ref{X}), we see that this is possible only if the denominator in (\ref{X}) vanishes. This vanishing condition will be determined by (\ref{D}) or (\ref{D1}).\\

\underline{Case 1: $b<1/2.$}\\

We have $c \in (-1,0)$ and  (\ref{D1}) then implies that $A^2 + B^2 < 4(1-b)^2,$ so that $\frac{2(1-b)}{\sqrt{A^2 + B^2}}>1.$ Therefore, the denominator in (\ref{X}) never vanishes and, since it is positive at $t=0,$ it is positive for all $t.$ This observation reconfirms the earlier statement made on the boundedness of the solutions.\\

\underline{Case 2: $b>1/2,b \neq 1.$}\\

Now, since $c>0,$ (\ref{D}) implies $\frac{2|1-b|}{\sqrt{A^2 + B^2}}<1.$ Therefore the denominator in (\ref{X}) now vanishes at finite  $t_j, \, j \in \mathbb{Z},$ which are the roots of the trigonometric equation
\[ B \cos t\,-\,A \sin t \,=\,2(b-1).\]
We may fix $t_0$ to be the smallest positive solution.
Thus, $x(t_j)$ is an odd integer multiple of $\pi /2$ and from the expression for $\dot{x},$ we see that $\dot{x}(t_j )\,=\,1.$ These values precisely correspond to the singularities in the equation (\ref{1}) or (\ref{2}).

We now write down an explicit expression for the solution by an integration of the expression (\ref{xd}) for $\dot{x}.$\\

Note that the integral of the secod term in the expression for $\dot{x}$ involves $\arctan$ function containg $\tan$ function in arguments. Therefore, a little care is required to write this expression. Let $\psi (t)$ denote the (indefinite) integral of this second term. We have
\begin{equation}
\psi (t)=
\begin{cases}
\arctan \left(\frac{[A^2 + (B-2(1-b))^2] \tan (t/2) - 4(1-b)A}{4 |2b-1| \cos^2 a} \right),   &\text{if $|t| < \pi$ } \\
\pm \pi /2    &\text{if $ t= \pm \pi$ }
\end{cases}
\end{equation}

We then extend the definition for a $t$ in an interval $[(2n-1) \pi, (2n+1) \pi ],\,n \in \mathbb{Z} $ by
\[\psi (t)\,=\,n \pi + \psi (t-2n \pi).\]
The solution is then expressed as

\begin{equation}
x(t)\,=\,a + \frac{t}{2} \pm (\psi (t) - \psi (0)) 
\end{equation}

\noindent for any real $t$, with positive sign if $b>1/2$ and negative sign if $b<1/2.$ This immediately gives periodicity of the solution when $b<1/2.$\\

In \cite{ss05}, the initial values $a,b$ are taken as
\[ a\,=\,0, \quad b\,=\,\frac{\xi}{\xi + 1},\,\,\xi \neq 0, \pm 1. \]

\noindent These values make the expressions above much simpler and we note them below.\\
We have $A=0$ and $B=2b= \frac{2 \xi}{\xi +1},$ and $\psi$ is given by $\psi (t)\,=\,\arctan \left(\frac{|\xi -1|}{|\xi +1|} \tan (t/2)\right) ~\mbox{for}~ |t|< \pi ~\mbox{and}~ \psi(\pm \pi)\,=\,\pm \pi /2.$ $\psi$ is then extended for all real $t$ as before. The expressions for $ \dot{x}, ~\mbox{and}~ x$ are given by
\begin{align}
\dot{x}(t)\,&=\,\frac{1}{2}\left(1 + \frac{\xi^2 - 1}{ \xi^2 + 2 \xi \cos t + 1}\right) \notag \\
x(t)\,&=\, t/2 \pm \psi (t),\,t \in \mathbb{R} \notag
\end{align}

\noindent In the  expression for $x$, $'+'$ is chosen when $|\xi| >1$ and $'-'$ is chosen when $|\xi| < 1.$ These expressions may be used to obtain more information about the integral mentioned in \cite{ss05}.\\

We remark that the phase portraits may also be easily drawn using the above expressions for $x$ and $\dot{x},$ which give a parametric representation of the orbit. A few orbits are depicted in the following two figures with $a=0, b=1/4 ~\mbox{and}~ -2/5$ (Figure $1$) and $a=0, b=3/4 ~\mbox{and}~3/2$ (Figure $2$) respectively. In these figures $y=\dot{x}.$\\
\begin{figure}[!h]
\center
\includegraphics[width=5in,height=4in,angle=0]{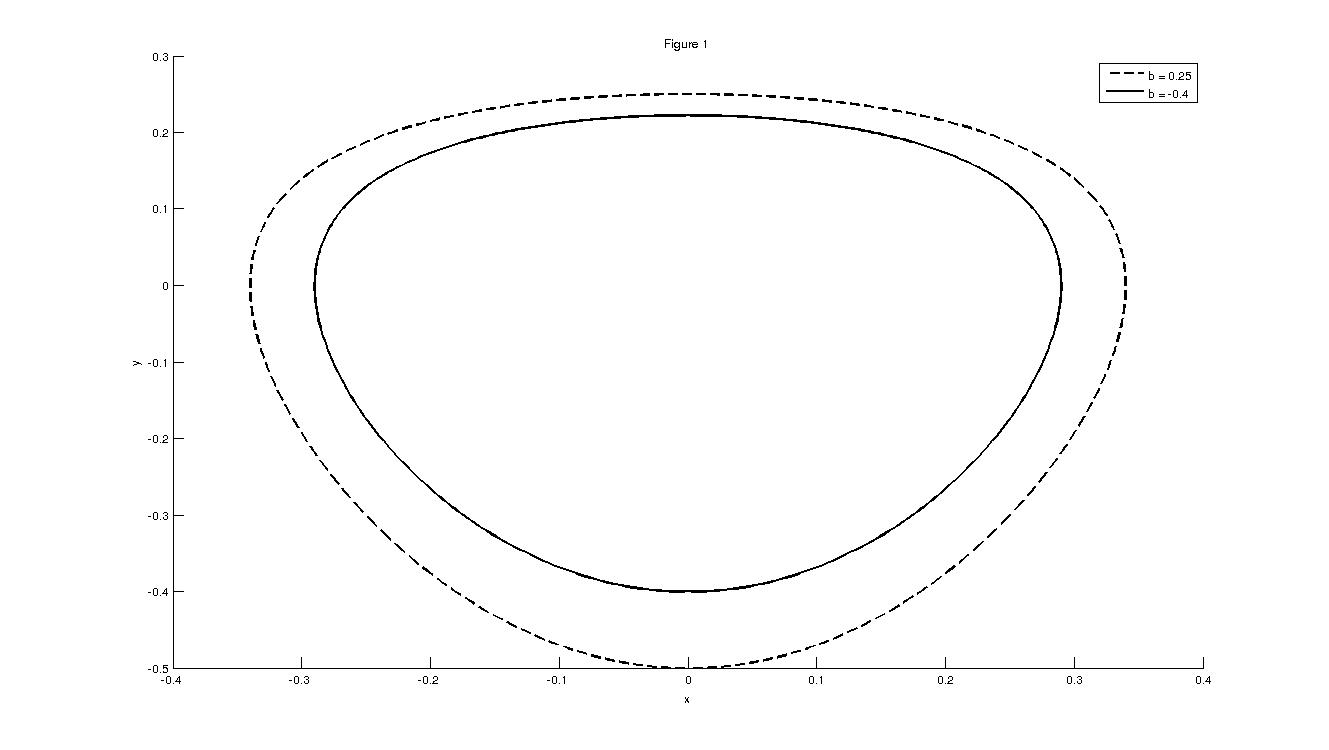}
\label{Fig1}
\end{figure}
\begin{figure}[!h]
\center
\includegraphics[width=5in,height=4in,angle=0]{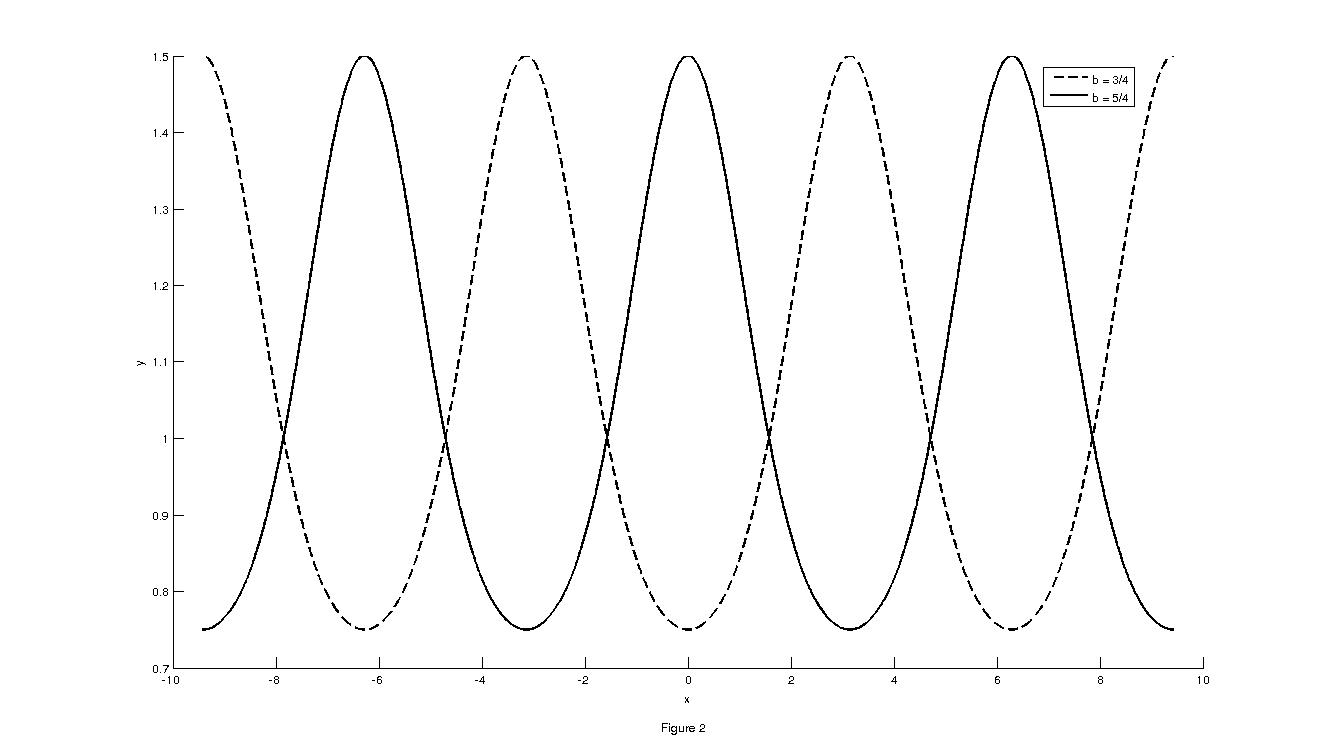}
\label{Fig2}
\end{figure}

\noindent {\bf Remark:} For the differential equation
\[ \frac{d}{dt}\left(\frac{\cos x}{1-\dot{x}}\right)\,=\,\sin x, \]
the energy equation is given by
\[ (1-\dot{x}) \exp (\dot{x})\,=\,c \cos x, \]
with $c$ a constant expressed in terms of the initial data. Using this energy equation we can do a similar phase plane analysis as above and obtain a complete phase portrait. Since an expression for $\dot{x}$ cannot be obtained, no further integration of the energy equation is possible and as such we will not be able to obtain the solutions in explicit form. The equilibrium points are again $(n \pi,0)$ with $n$ an integer. A simple linear analysis shows that each equilibrium point is linearly unstable and is of saddle type. Thus, by Hartman-Grobman theorem, the same behaviour locally persists for the nonlinear equation also. One can similarly prove the smoothness of the solutions; in fact, real analyticity of the solutions can be established using Gevrey-1 estimates.

\end{document}